\documentclass[11pt]{amsart}
\usepackage{amsmath,enumerate,amsfonts,amssymb,amsthm, verbatim,esint}

\newtheorem{theorem}{Theorem}
\newtheorem{lemma}[theorem]{Lemma}
\newtheorem{proposition}[theorem]{Proposition}
\newtheorem{corollary}[theorem]{Corollary}

\theoremstyle{definition}
\newtheorem{definition}[theorem]{Definition}

\newtheorem{notrems}[theorem]{Notation and Remarks}

\newtheorem{remark}[theorem]{Remark}

\newcommand{\Section}[1]{\section{#1}\setcounter{theorem}{0}}

\begin{document}

\title[Heat coefficients of surfaces with curved conical singularities]{Heat coefficients of surfaces with curved conical singularities}
\author{Dorothee Schueth}
\address{Institut f\"ur Mathematik, Humboldt-Universit\"at zu
Berlin, 10099 Berlin, Germany}
\email{dorothee.schueth@hu-berlin.de}


\begin{abstract}
Let $(M,g)$ be a two-dimensional Riemannian manifold of finite diameter with a conical singularity. Under the assumption that the metric near the
cone point~$C$ is rotationally invariant, but not necessarily flat, we give an explicit formula for the coefficient $b_{1/2}(C)$ in the heat trace expansion $\operatorname{tr}(\operatorname{exp}(-t\Delta_g))\sim_{t\searrow0}
(4\pi t)^{-1}\sum_{j=0}^\infty a_j(M) t^j+\sum_{j=0}^\infty b_{j/2}(C)t^{j/2}+\sum_{j=0}^\infty c_{j/2}(C) t^{j/2} \log t$.
In the case that the Gaussian curvature~$K$ of $(M,g)$ satisfies $|K(p)|\to\infty$ as $p\to C$,
we show that $b_{1/2}(C)$ varies irrationally under constant rescalings of the distance
circles near the cone point. This is a sharp contrast to the behavior of $b_0(C)$ and of those coefficients $b_j(C)$
which appear in certain known formulas in the case of orbifold cone points or corners of geodesic polygons.
\end{abstract}

\maketitle

\Section{Introduction}
\label{sec:intro}
\noindent
After the seminal work of Cheeger~\cite{Ch83} on extending the theory of the Laplace operator to Riemannian spaces with singularities,
the asymptotic behaviour of the resolvent trace and the heat trace on manifolds with conical singularities (or, more generally, stratified spaces with a stratum of conical type) was studied, among others, by Br\"uning and Seeley \cite{BS87}, \cite{BS91} who considered certain associated one parameter families of operators using a functional analytic approach. 
Their theory was extended to stratified spaces with iterated cone-edge metrics by Hartmann, Lesch, and Vertman \cite{HLV18}, \cite{HLV21}.
In the case of two-dimensional Riemannian manifolds with isolated conical singularities, the metric near such a singularity~$C$ has the form
\begin{equation}
\label{eq:g}
g=dr^2+f(r)^2 d\theta^2,
\end{equation}
where $d\theta^2$ is the standard metric on the circle~$S^1$ of length~$2\pi$, and where $f$ is a smooth function on some $[0,\varepsilon)$ with $f(0)=0$ and $f'(0)>0$. The special case $f(r)=\textrm{const}\cdot r$ corresponds to a cone point in the classical sense, where the metric near the singularity is flat.

Assume that the conical singularity $C$ is the only singularity of the surface~$M$ and that the closure~$\overline M=M\cup\{C\}$ of $M$ with respect to the Riemannian distance is compact. We denote by~$\Delta$ the Friedrichs extension of the Laplace operator on functions on~$M$. General results from~\cite{BS87} imply that in this setting, the associated heat trace has an asymptotic expansion
\begin{equation}
\label{eq:heatasy1}
\operatorname{tr}(\exp(-t\Delta))\sim{\textstyle{\frac1{4\pi t}}}\sum_{j=0}^\infty a_j(M) t^j+\sum_{j=0}^\infty b_{j/2}(C) t^{j/2}+\sum_{j=0}^\infty c_{j/2}(C) t^{j/2}\log t,
\end{equation}
for $t\searrow0$, where the coefficients in the first sum are (possibly regularized) integrals over certain curvature invariants on~$M$, while the second and third sums depend only on the germ of~$f$ at $r=0$ and correspond to the contribution of the cone point $C$; see Section~\ref{sec:prelim} for more details.
It is well-known that $a_0(M)=\operatorname{vol}(M)$. Moreover, $b_0(C) = \frac1{12}(\frac1{f'(0)}-f'(0))$ (see~\cite{BS87}, p.~424), and, in our situation, $c_0(C) =0$ (see~\cite{BS87}, p.~423/424) and $c_{j/2}(C)=0$ for all odd~$j$ (see, e.g., Remark~\ref{rem:c-coeffs} below).

In her dissertation~\cite{Su17}, Suleymanova computed $b_0(C)$ and $c_0(C)$ for cone points~$C$ of higher dimensional manifolds
under the assumption that the metric near the cone point is of the special form $g=dr^2+r^2 g_N$ on a punctured neighborhood $U\approx (0,\varepsilon)\times N$ of~$C$, where $g_N$ is a Riemannian metric on the cross section~$N$. 

Note that cone points of orbisurfaces constitute a special case of conical singularities -- at least if one assumes full rotational symmetry of a punctured neighborhood of the cone point, such that it fits into the above setting. Dryden et al.~\cite{DGGW} applied general results by Donnelly to derive a qualitative description of the heat trace expansion for compact orbifolds. No logarithmic terms occur in that case. For cone points~$C$ of order~$n$ in two-dimensional Riemannian orbifolds, it was shown in~\cite{DGGW} that $b_1(C)=\big(\frac1{360}(n^3-\frac1n)+\frac1{36}(n-\frac1n)\big)K(p)$, where $K$ is the Gauss curvature and $p$ the point corresponding to~$C$ in an orbifold chart. In~\cite{Sc19}, the author obtained a similar formula for $b_2(C)$ in the same context; $b_2(C)$ turns out to be linear combination of~$K(p)^2$ and $(\Delta K)(p)$, where the coefficients are, again, rational functions of the order of the cone point.

The cited results for $b_1(C)$ and $b_2(C)$ in the orbifold case do not assume full rotational symmetry of a punctured neighborhood of~$C$.

U\c car \cite{Uc17} obtained explicit formulas for \emph{all} $b_\ell(C)$ for cone points~$C$ of two-dimensional orbifolds under the special assumption that the orbifold has constant curvature~$\kappa\in\mathbb{R}$.
In our notation, this corresponds to the case $f(r)=\frac1n s_\kappa(r)$ in~\eqref{eq:g}, where $n$ is the order of the cone point and $s_\kappa$ is the modified sine function determined by $s_\kappa''=-\kappa s_\kappa$, $s_\kappa(0)=0$, and $s_\kappa'(0)=1$. More specifically, $b_\ell(C)$ can then be written as $\kappa^\ell$ times $\frac1n p_\ell(n)$, 
where $n$ is the order of the cone point and $p_\ell$ is a certain polynomial of order $2\ell+2$. U\c car~\cite{Uc17} also proved formulas of the same type for the contributions of corners of geodesic polygons in surfaces of constant curvature to the Dirichlet heat trace expansion, regardless of whether the angle at the corner is of the form $\pi/n$ or not.

We note here, without proving it in this paper, that in the case of full rotational symmetry around~$C$, it would be possible to reprove the formulas for $b_1(C)$ from~\cite{DGGW} and for $b_2(C)$ from~\cite{Sc19} (or, for the constant curvature case, from~\cite{Uc17}) using similar methods as in the current paper, i.e. relying on the approach of~\cite{BS87}. For this, it plays no role at all whether $n=1/f'(0)>0$ is a natural number or not. Therefore, replacing $n$ by $1/f'(0)$,
we can note that these $b_\ell(C)$ depend rationally on~$f'(0)$. Equivalently, under rescalings of~$f$ by a constant $\lambda>0$, these coefficients depend rationally on~$\lambda$. (Note that the Gaussian curvature $K=-f''/f$ and its derivatives, which also appear in the $b_\ell(C)$ above, are invariant under constant rescalings of~$f$.)

If $p$ is a point in a \emph{smooth} two-dimensional Riemannian manifold~$M$ and the metric on a neighborhood of~$p$ has full rotational symmetry around~$p$, then the metric near $p$ can be written in the form~\eqref{eq:g}, where $r=d(p,\,.\,)$ and $f(r)$ is the distortion of the length of the distance circle under the geodesic exponential map~$\exp_p$\,. In that case, it is well-known that $f(r) = r -\frac16 K(p)r^3 + O(r^4)$; in particular, $f''(0)=0$. Passing to an orbifold cone point of order~$n$ just corresponds to passing from $f$ to $f/n$, so $f''(0)=0$ still holds for cone points in orbisurfaces. The same, of course, holds for arbitrary constant rescalings of~$f$.

In~\eqref{eq:g}, however, we allow more general functions~$f$, as long as they are smooth on $[0,\varepsilon)$ for some $\varepsilon>0$ and satisfy $f(0)=0$ and $f'(0)>0$. In particular, we allow $f''(0)\ne0$. Since $K=-f''/f$ in the coordinates from~\eqref{eq:g}, the inequality $f''(0)\ne0$ is equivalent to $|K(p)|\to\infty$ for $p\to C$. On the other hand, if a punctured neighborhood of the singularity happens to be isometric to a rotational surface embedded in~$\mathbb{R}^3$ then $f'(0)=\sin\varphi$, where~$\varphi$ is the sine of the angle between the axis and the initial direction of the profile curve; if $\varphi<\pi/2$ then the inequality $f''(0)\ne0$ is equivalent to nonzero curvature of the profile curve at its initial point (see Remark~\ref{rem:embedding}).

If $f''(0)\ne0$ in~\eqref{eq:g}, then, unlike in the orbifold case, half powers of~$t$ can occur in the middle sum of~\eqref{eq:heatasy1}, and logarithmic terms can occur, too.
For example, it turns out that $c_1(C)=-\frac1{60} f''(0)^2/f'(0)$ (see Remark~\ref{rem:c-coeffs}).
The main goal of this paper is to prove the following explicit formula for $b_{1/2}(C)$:
\begin{equation}
\label{eq:bhalftotintro}
b_{1/2}(C)
=\frac{2 f''(0)}{\sqrt\pi f'(0)}\int_0^1\left(\hat h_{2,\alpha}(1-u^2)-\frac14\hat h_{0,\alpha}(1-u^2)\right)\,du,
\end{equation}
where $\alpha:=1/f'(0)$ and
$\hat h_{k,\alpha}$ is defined as follows for all $k\in\mathbb{N}_0$ and $\alpha>0$: For $z\in\mathbb{C}$ with $|z|<1$ and $|1-z|<1$, first let
$$h_{k,\alpha}(z):=\sum_{n=0}^\infty\alpha^k n^k(1-z)^{n\alpha}=\left(-(1-z)\frac d{dz}\right)^k\frac1{1-(1-z)^\alpha}\,.
$$
This turns out to define a meromorphic function in a neighborhood of the origin, with a pole in $z=0$. Denote its regular part (obtained by subtracting the principal part of the Laurent series) by $h^{\textrm{reg}}_{k,\alpha}(z)$ and let
$$\hat h_{k,\alpha}(z):=\frac1z(h^{\textrm{reg}}_{k,\alpha}(z)-h^{\textrm{reg}}_{k,\alpha}(0));
$$
see Section~\ref{sec:formulab1} for a detailed discussion of these functions.

Our formula for $b_{1/2}(C)$ differs fundamentally from the one which is stated in~\cite{Su17-2} and claims this coefficient to be equal to $\frac{f''(0)}{f'(0)}\cdot\frac5{96\sqrt\pi}$. (More precisely, $1/f'(0)$ is written on p.~10 of that paper as $\cot\varphi$ instead of $1/\sin\varphi$, where $\varphi$ is the opening angle in the case of embedded rotational surfaces as above; however, this small mistake is not the main reason for the difference between the formula from~\cite{Su17-2} and our result.)
See Remark~\ref{rem:polespec} for more details concerning this discrepancy. 

Moreover, in contrast to the situation for the $b_\ell(C)$ as described above, it turns out that if $f''(0)\ne0$ then $b_{1/2}(C)$ does \emph{not} depend rationally on the scaling factor~$\lambda$ when one replaces $f$ by $\lambda f$ -- or equivalently, when one rescales small distance circles around~$C$ by~$\lambda$. In fact, the integral in~\eqref{eq:bhalftotintro} turns out to be a non-rational function of~$\alpha=1/f'(0)$. To the author's knowledge, this is the first instance where such an irrational behaviour of a coefficient in the heat expansion of conical singularities is detected.

This paper is organized as follows: In Section~\ref{sec:prelim} we recall, adapted to our special setting, the necessary background from~\cite{BS87} concerning the description of~$\Delta$ near the singularity in terms of a certain one parameter family of operators $A(r)$, the associated scaled boundary operators and the expansion of the trace of their resolvents, thereby arriving at a first description of the contributions of the singularity to the coefficients of the asymptotic expansion of the resolvent trace. We also recall the relations between these and the corresponding coefficients of the heat trace expansion.
In Section~\ref{sec:formulab1} we derive our explicit formula for~$b_{1/2}(C)$ (Corollary~\ref{cor:b1total}), and
in Section~\ref{sec:irrat} we prove that if $f''(0)\ne0$ then $b_{1/2}(C)$ is $f''(0)/f'(0)$ times an irrational function of~$f'(0)$ (Theorem~\ref{thm:irrat}).
\noindent

\Section{Preliminaries}
\label{sec:prelim}

\begin{notrems}
\label{notrems:Ccoord}
\begin{itemize}
\item[(i)]
In this paper, $(M,g)$ will always denote a two-dimensional Riemannian manifold of finite diameter with one conical singularity~$C$.
We assume that the closure of~$M$ with respect to the Riemannian distance is compact and equals $\overline M=M\cup\{C\}$.
\item[(ii)]
By definition of the notion of conical singularity, there exists $\varepsilon>0$ such that the punctured $\varepsilon$-neighborhood
$U_\varepsilon$ of~$C$ is isometric to $(0,\varepsilon)\times S^1$ equipped with the metric
$$dr^2+f(r)^2 d\theta^2
$$
for some $f\in C^\infty([0,\varepsilon),\mathbb R_{\ge0})$ with $f(0)=0$ and $f'(0)>0$. Here, $d\theta^2$ denotes the standard metric on~$S^1$ with length~$2\pi$.
\end{itemize}
\end{notrems}

\begin{remark}
In the above coordinates, the Gaussian curvature $K$ is given by
$$K=-f''/f.
$$
Therefore, for $p\to C$ we have
\begin{equation*}
K(p)\to\begin{cases}-f'''(0)/f'(0),&f''(0)=0,\\ \infty,&f''(0)<0,\\ -\infty,&f''(0)>0.\end{cases}
\end{equation*}
In particular, $f''(0)\ne0$ is equivalent to $|K(p)|\to\infty$ for $p\to C$.
\end{remark}

\begin{remark}
\begin{itemize}
\label{rem:embedding}
\item[(i)]
In the special case that $U_\varepsilon$ can be embedded isometrically as a surface of revolution around the $x$-axis in~$\mathbb{R}^3$ with profile curve $(0,\varepsilon)\ni r\mapsto
(x(r),0,z(r))\in\mathbb{R}^3$ of unit speed with $x(r)\ge0$ and $z(r)\ge0$ for all~$r$, one has $f(r)=z(r)$, ${x'}^2+{f'}^2=1$, and $x'x''+f'f''=0$. In particular, $f'(0)\le1$.
\item[(ii)]
If $f'(0)<1$ in the situation of~(i), then $x'(0)>0$, and the curvature $\kappa(0)$ of the profile curve in the $(x,z)$-plane in its initial point is given by
$$\kappa(0)=(x'f''-f'x'')(0)=\frac{({x'}^2f''+(f')^2f'')(0)}{x'(0)}=\frac{f''(0)}{x'(0)}=\frac{f''(0)}{f'(0)}\cdot\tan\varphi,
$$
where $\varphi=\arctan(f'(0)/x'(0))=\arcsin(f'(0))\in(0,\frac\pi2)$ is the angle between the $x$-axis and the initial direction of the profile curve.
In particular, $f''(0)\ne0$ is equivalent to $\kappa(0)\ne0$ in this situation.
\end{itemize}
\end{remark}

\begin{remark}
\label{rem:laplace}

(i)
Following~\cite{BS87}, p.~423, we note that in the setting of \ref{notrems:Ccoord}(ii) the associated Laplacian
$$\Delta=-f^{-1}(\partial_r f\partial_r)-f^{-2}\partial_\theta^2
$$
on~$U_\varepsilon$ is conjugate, via the map $u\mapsto f^{1/2} u$, to the operator
\begin{equation*}
\begin{aligned}
\label{eq:opl}
L:=&{}-\partial_r^2+f^{-2}\left(-\partial_\theta^2-\frac14(f')^2+\frac12 ff''\right)\\
=&{}-\partial_r^2+r^{-2}A(r)
\end{aligned}
\end{equation*}
on $(0,\varepsilon)\,\times S^1$, where each
\begin{equation*}
A(r):=-\frac{r^2}{f(r)^2}\partial_\theta^2-\frac14\frac{r^2f'(r)^2}{f(r)^2}+\frac12\frac{r^2 f''(r)}{f(r)}
\end{equation*}
acts on $C^\infty(S^1)$.
Note that $A$ smoothly extends  to $r=0$.

(ii)
We have
$$r^2/f(r)^2=(f'(0)+\frac r2 f''(0)+O(r^2))^{-2}=f'(0)^{-2}(1-rf''(0)/f'(0)+O(r^2))
$$
and, thus,
$$r^2f'(r)^2/f(r)^2=(1-rf''(0)/f'(0))(1+rf''(0)/f'(0))^2+O(r^2)=1+rf''(0)/f'(0)+O(r^2).
$$
Moreover,
$$r^2f''(r)/f(r)=rf''(0)/f'(0)+O(r^2).
$$ 
Hence,
\begin{equation*}
\begin{aligned}
 A(r)=&{}-\frac1{f'(0)^2}\left(1-r\frac{f''(0)}{f'(0)}+O(r^2)\right)\partial_\theta^2
      -\frac14\left(1+r\frac{f''(0)}{f'(0)}\right)+\frac12 r\frac{f''(0)}{f'(0)}+O(r^2)\\
			=&{}-\frac1{f'(0)^2}\left(1-r\frac{f''(0)}{f'(0)}+O(r^2)\right)\partial_\theta^2
			-\frac14\left(1-r\frac{f''(0)}{f'(0)}+O(r^2)\right).
    \end{aligned}
\end{equation*}
In particular,
\begin{equation}
\label{eq:A0}
A(0)=-\frac1{f'(0)^2}\partial_\theta^2-\frac14\ge-\frac14
\end{equation}
and
\begin{equation}
\label{eq:Aprime0}
A'(0)=k_f A(0)\mbox{ with }k_f:=-\frac{f''(0)}{f'(0)}\,.
\end{equation}
\end{remark}

\begin{remark}
\label{rem:resolv}
We denote by~$\Delta$ the Friedrichs extension of the Laplacian on the singular surface~$M$. 
For $m>\dim(M)/2=1$, the resolvent $(\Delta+z^2)^{-m}$ is of trace class; see, e.g.~\cite{BS91}, p.~275. Moreover, there is an asymptotic expansion as $z\to\infty$:
\begin{equation}
\label{eq:resolvasy}
\operatorname{tr}(\Delta+z^2)^{-m}\sim\sum_{j=0}^\infty a_{j,m}(M) z^{-2m+2-2j}+\sum_{j=0}^\infty b_{j,m}(C) z^{-2m-j}+\sum_{j=0}^\infty c_{j,m}(C)z^{-2m-j}\log z,
\end{equation}
where the coefficients in the first sum are (possibly regularized) integrals over certain curvature invariants on~$M$, while the second and third sums depend only on the germ of~$f$ at $r=0$ and consist of contributions of the cone point $C$.
This follows (using $\operatorname{dim}(M) = 2$ and $\operatorname{dim}(C) = 0$) as a special case of either \cite{BS91}, Theorem~5.2,  or~\cite{HLV21}, Theorem~1.1,  each with slightly different notation.
\end{remark}

\begin{corollary}
The heat trace associated with~$\Delta$ has the asymptotic expansion
\begin{equation}
\label{eq:heatasy}
\operatorname{tr}(\exp(-t\Delta))\sim{\textstyle{\frac1{4\pi t}}}\sum_{j=0}^\infty a_j(M) t^j+\sum_{j=0}^\infty b_{j/2}(C) t^{j/2}+\sum_{j=0}^\infty c_{j/2}(C) t^{j/2}\log t,
\end{equation}
for $t\searrow0$,
where the relation to the coefficients of~\eqref{eq:resolvasy} is as follows:
\begin{equation}
\label{eq:coeffsrel}
\begin{aligned}
a_j(M)&=\frac{4\pi(m-1)!}{\Gamma(m-1+j)}a_{j,m}(M)\\
b_{j/2}(C)&=\frac{(m-1)!}{\Gamma(m+\frac j2)}b_{j,m}(C) + \frac{(m-1)!\,\Gamma'(m+\frac j2)}{2\Gamma(m+\frac j2)^2}c_{j,m}(C)\\
c_{j/2}(C)&= - \frac{(m-1)!}{2\Gamma(m+\frac j2)}c_{j,m}(C)
\end{aligned}
\end{equation}
for each $j\in\mathbb{N}_0$.
\end{corollary}

\begin{proof}
The existence, in itself, of the asymptotic expansion~\eqref{eq:heatasy} is well-known; e.g., it is a special case of \cite{BS87}, Theorem~7.1. The explicit relations~\eqref{eq:coeffsrel} to the coefficients in the resolvent expansion are a bit hidden in the formulas on p.~416 of that (more general) paper. The argument is as follows:
Let $\gamma\subset\mathbb{C}$ denote the contour $\{\lambda:\operatorname{arg}(\lambda+1)=\pm\pi/4\}$, traversed upward. Then
\begin{equation*}
\operatorname{tr}e^{-t\Delta}=t^{1-m}\frac{(m-1)!}{2\pi i}\int_\gamma e^{-t\lambda}\operatorname{tr}(\Delta-\lambda)^{-m}\,d\lambda.
\end{equation*}
We insert~\eqref{eq:resolvasy} with $z=\sqrt{-\lambda}$ and use the equations
\begin{equation}
\label{eq:cauchytrans}
\begin{aligned}
\int_\gamma e^{-t\lambda}(-\lambda)^{-n/2}d\lambda&=\frac{2\pi i}{\Gamma(n/2)}t^{-1+n/2}\\
\int_\gamma e^{-t\lambda}(-\lambda)^{-n/2}\log\sqrt{-\lambda}\,d\lambda&=
- \frac{\pi i}{\Gamma(n/2)}t^{-1+n/2}\log t + \frac{\pi i\,\Gamma'(n/2)}{\Gamma(n/2)^2}t^{-1+n/2}
\end{aligned}
\end{equation}
 for all $t>0$ and $n\in\mathbb{N}$. Comparing coefficients then yields~\eqref{eq:coeffsrel}.
\end{proof}

\begin{remark}
\label{rem:naturebcm}
We will now summarize some facts, mainly from~\cite{BS87}, in order to establish formulas for the resolvent expansion coefficients $b_{j,m}(C)$ and $c_{j,m}(C)$ in our specific situation.

(i)
Recall from Remark~\ref{rem:laplace} that $\Delta$ on 
$$U_\varepsilon\simeq (0,\varepsilon)\times S^1
$$
is conjugate to
$$
L=-\partial_r^2+r^{-2}A(r).
$$
We choose a modification~$\tilde A$ of~$A$ as in~\cite{BS87}, p.~396;
in particular, $\tilde A$ is defined on all of $[0,\infty)$, coincides with~$A$ in some neighborhood of $r=0$, and $\tilde A(r)$ is an elliptic operator on $C^\infty(S^1)$ for each~$r$.
The corresponding modification of~$L$ is the ``boundary operator''
$$L_b:=-\partial_r^2+r^{-2}\tilde A(r).
$$
For any $t\geq 0$, let
$$\tilde A_t(r):=\tilde A(tr)
$$
and consider the ``scaled boundary operator''
$$L_{b,t}:=-\partial_r^2+r^{-2}\tilde A_t(r).
$$
Writing
$$G^m_{b,t}(z):=(L_{b,t}+z^2)^{-m},
$$
we recall from~\cite{BS87}, Lemma~4.9, that for each $m>\operatorname{dim}(M)/2=1$, the operators $G^m_{b,t}(z)$ have kernels $G^m_{b,t}(r,s;z)$ such that each $G^m_{b,t}(\,.\,,\,.\,;z)$ is a continuous map from
$(0,\infty)\times(0,\infty)$ into the space of trace class operators on
$$H:=L^2(S^1,d\theta^2),
$$
and then $G_{b,t}^m$ satisfies
\begin{equation}
\label{eq:trafo}
G_b^m(tr,ts;z/t)=t^{2m-1}G_{b,t}^m(r,s;z),
\end{equation}
where
$$G_b:=G_{b,1}.
$$
Moreover, by~\cite{BS87}, Theorem~4.2,
 $\varphi G^m_{b,t}(z)$ is of trace class for each $\varphi\in C^\infty([0,\infty),\mathbb{R})$ with compact support,
and
\begin{equation*}
\operatorname{tr}(\varphi G^m_{b,t}(z))=\int_0^\infty \varphi(r)\operatorname{tr}_H G^m_{b,t}(r,r;z)\,dr.
\end{equation*}
In particular,
$$\operatorname{tr}(\varphi G^m_b(z))=\int_0^\infty\sigma(r,rz)\,dr
$$
with
\begin{equation*}
\sigma(r,\zeta):=\varphi(r)\operatorname{tr}_H G^m_b(r,r;\zeta/r).
\end{equation*}
By formula~(4.18) in~\cite{BS87},
\begin{equation}
\label{eq:sigmatrafo}
\sigma(r,\zeta)=\varphi(r)\cdot r^{2m-1}\operatorname{tr}_H G^m_{b,r}(1,1;\zeta),
\end{equation}
which is a consequence of~\eqref{eq:trafo}.
Since the operators $L_{b,r}$ are elliptic with resolvents $G_{b,r}$, it follows that there is an asymptotic expansion of the form
\begin{equation}
\label{eq:sigmaasy}
\sigma(r,\zeta)\sim \sum_{j=0}^\infty \varphi(r)\sigma_j(r)\zeta^{-2m+\operatorname{dim}(M)-2j}=\sum_{j=0}^\infty\varphi(r)\sigma_j(r)\zeta^{-2m+2-2j}
\end{equation}
for $\zeta\to\infty$ (see, e.g., \cite{BS87}, p.~421, or \cite{BS91}, p.~286).

(ii)
{}From now on, we let $\varphi$ be a smooth cut-off function on~$[0,\infty)$ with compact support and
$$\varphi(r)\equiv 1\mbox{\ near\ }r=0;
$$
in particular, all derivatives of~$\varphi$ in $r=0$ vanish.
We also denote by~$\varphi$ the corresponding function on~$U_\varepsilon$ and its extension to $M$ by zero outside~$U_\varepsilon$.
The asymptotic expansion~\eqref{eq:sigmaasy} and the Singular Asymptotics Lemma (see \cite{BS87}, p.~372) now imply
\begin{equation}
\label{eq:singasy}
\begin{aligned}
\int_0^\infty \sigma(r,rz)\,dr\sim{}&\sum_{j=0}^\infty \fint_0^\infty r^{-2m+2-2j}\varphi(r)\sigma_j(r)\,dr\cdot z^{-2m+2-2j}\\
&+\sum_{k=0}^\infty \fint_0^\infty\frac{\zeta^k}{k!}{\partial_r^k}_{|r=0}\,\sigma(r,\zeta)\,d\zeta\cdot z^{-k-1}\\
&+\sum_{\ell=0}^\infty \frac1{(2m-3+2\ell)!}\sigma_\ell^{(2m-3+2\ell)}(0) \cdot z^{-2m+2-2\ell} \log z
\end{aligned}
\end{equation}
for $z\to\infty$.
Here, $\fint$ denotes the regularized integral:
$$\fint_0^\infty \mu(r)\,dr:={\operatorname{Res}_0}_{|s=1}\,\mathcal{M}\mu,
$$
where $\mathcal{M}\mu:s\mapsto\int_0^\infty r^{s-1} \mu(r)\,dr$ is the Mellin transform of~$\mu$,
and ${\operatorname{Res}_0}_{|s=1}\,\mathcal{M}\mu$ is the coefficient at $(s-1)^0$ (i.e., the constant term) in the Laurent expansion around $s=1$ of a meromorphic continuation of~$\mathcal{M}\mu$; see~\cite{Le97}, Section~2.1 for a thorough introduction to these concepts. Note that, equivalently,
\begin{equation}
\label{eq:regint}
\fint_0^\infty \mu(r)\,dr={\operatorname{Res}_0}_{|s=0}\int_0^\infty r^s \mu(r)\,dr.
\end{equation}

(iii)
The coefficients $b_{j,m}(C)$ and $c_{j,m}(C)$ in the asymptotic expansion~\eqref{eq:resolvasy} of $\operatorname{tr}(\Delta+z^2)^{-m}$ for $z\to\infty$  are the same as in the asymptotic expansion of
$\operatorname{tr}(\varphi(\Delta+z^2)^{-m})$ which, in turn, are the same as in the asymptotic expansion
of $\operatorname{tr}(\varphi G_b^m(z))=\int_0^\infty\sigma(r,rz)\,dr$ (see, e.g., \cite{BS91}, p.~284).
{}From~\eqref{eq:singasy} it follows that
\begin{equation}
\label{eq:b-explicit}
\begin{aligned}
b_{j,m}(C)&=\frac1{(2m-1+j)!}\cdot\fint_0^\infty \zeta^{2m-1+j}{\partial_r^{2m-1+j}}_{|r=0}\,\sigma(r,\zeta)\,d\zeta\\
&=\frac1{j!}\fint_0^\infty\zeta^{2m-1+j}\,{\partial_r^j}_{|r=0}\,\operatorname{tr}_H G_{b,r}^m(1,1;\zeta)\,d\zeta,
\end{aligned}
\end{equation}
where for the second equation one uses~\eqref{eq:sigmatrafo} and $\varphi(r)\equiv 1$ near $r=0$,
and
\begin{equation}
\label{eq:c-explicit}
c_{j,m}(C)=\begin{cases}
       0,&j\mbox{ odd},\\
			\frac1{(2m-1+j)!}\cdot\sigma_{1+\frac j2}^{(2m-1+j)}(0),&j\mbox{ even}.\end{cases}
\end{equation}

(iv)
Note that for $r$ sufficiently close to~$0$, the asymptotic expansion of $\sigma(r,\zeta)$ for $\zeta\to\infty$ is the same as that of the integral over $(S^1,\operatorname{dvol}_{f^2d\theta^2})$
of the value of the operator kernel of $(\Delta+\zeta^2/r^2)^{-m}$ in $(p,p)$, where $p=(r,\theta)$.
Consider the well-known the asymptotic expansion
$$H(t,p,p)\sim\frac 1{4\pi t}\sum_{\ell=0}^\infty u_\ell(p)t^\ell
$$
for $t\to0$ of the heat kernel~$H$ of~$\Delta$, that is, the (interior) operator kernel associated with~$e^{-t\Delta}$.
Using the first formula of~\eqref{eq:cauchytrans} and the above, one obtains, for all sufficiently small $r>0$:
$$
\sigma(r,\zeta)\sim\sum_{j=0}^\infty \frac1{4\pi}(\zeta/r)^{-2m+2-2j}\int_{S^1}u_j(r,\theta)f(r)d\theta\cdot\frac{\Gamma(m-1+j)}{(m-1)!}
$$
for $\zeta\to\infty$.
Since $g$ on $U_\varepsilon$ is rotationally invariant, we can write $u_j(r):=u_j(r,\theta)$ and get
$$
\sigma(r,\zeta)\sim\sum_{j=0}^\infty \zeta^{-2m+2-2j}\cdot\frac12 r^{2m-2+2j} u_j(r)f(r)\cdot\frac{(m-2+j)!}{(m-1)!}
$$
for $\zeta\to\infty$; thus,
\begin{equation*}
\sigma_j(r)=\frac12 r^{2m-2+2j} f(r)u_j(r)\cdot\frac{(m-2+j)!}{(m-1)!}
\end{equation*}
for all $j\in\mathbb{N}_0$ and all sufficiently small $r>0$ (see also formula~(3.19) of~\cite{Su17} which concerned the special case $f(r)=r$ with general dimension of~$M$).
In particular, \eqref{eq:c-explicit} now implies:
\begin{equation}
\label{eq:c-explicit2}
c_{j,m}(C)\begin{cases} 0,&j\mbox{ odd},\\
\frac1{(2m-1+j)!}\cdot\frac12\cdot{\partial_r^{2m-1+j}}_{|r=0}\,( r^{2m+j} f(r) u_{1+\frac j2}(r))\cdot\frac{(m-1+\frac j2)!}{(m-1)!},&j\mbox{ even}.
\end{cases}
\end{equation}
\end{remark}

\begin{remark}
\label{rem:c-coeffs}
(i)
Recall that
\begin{equation*}
u_0(r)=1\mbox{\ and\ }u_1(r)=\frac16\operatorname{scal}(r)=\frac13 K(r)=-\frac13 f''(r)/f(r).
\end{equation*}
In particular, both $fu_0$ and $fu_1$ are smooth in $r=0$. By~\eqref{eq:c-explicit2}, this implies
$$c_{0,m}(C)=0
$$
(as well as $c_{-2,m}(C)=0$ for the hypothetical coefficient $c_{-2,m}(C)$). This is well-known; see, e.g., \cite{BS87}, p.~423/424. Since $c_{j,m}(C)=0$ for all odd~$j$ by \eqref{eq:c-explicit}, one concludes using~\eqref{eq:coeffsrel}: $c_0(C)=0$, $c_{1/2}(C)=0$, $b_0(C)=b_{0,m}(C)$, and
\begin{equation}
\label{eq:bhalf} 
b_{1/2}(C)= \frac{(m-1)!}{\Gamma(m+\frac12)}b_{1,m}(C).
\end{equation}
Moreover, $u_2(r) = \frac1{15}K(r)^2=\frac1{15}f''(r)^2/f(r)^2$. One easily derives using~\eqref{eq:c-explicit2}:
\begin{equation*}
c_{2,m}(C)=\frac m{30} f''(0)^2/f'(0)
\end{equation*}
and, consequently, $c_1(C)=-\frac1{60} f''(0)^2/f'(0)$. Note that the above statements about $c_{0,m}(C)$ and $c_{2,m}(C)$ do agree with those given in~Proposition~3.3 of \cite{Su17-2}, while our explicit formula for $b_{1,m}(C)$ in Corollary~\ref{cor:b1total} will differ fundamentally from the one given in Proposition~3.5 of~\cite{Su17-2}; see also Remark~\ref{rem:polespec}(i) below.

(ii)
Just for a moment, let us consider the special case that a punctured neighborhood of~$C$ in $\overline M=M\cup\{C\}$, after removing one radial geodesic, is isometric to the interior of a compact subset of a smooth surface. In that case,
the Gaussian curvature and, therefore, each of the functions $u_j$ smoothly extends to $r=0$. By~\eqref{eq:c-explicit2} this implies that all $c_{j,m}(C)$ vanish, and so do the $c_{j/2}(C)$. So, in that setting, no logarithmic terms occur in the asymptotic expansions \eqref{eq:resolvasy} and \eqref{eq:heatasy}.
For cone points in orbisurfaces this is, of course, well-known (even without the assumption of full rotational invariance of the metric); see, e.g., \cite{DGGW}.
\end{remark}

\Section{An explicit formula for $b_{1/2}(C)$}
\label{sec:formulab1}

\noindent
The explicit value of $b_0(C)=b_{0,m}(C)$ is well-known: By~\cite{BS87}, p.~424, interpreted in our notation (see also~\cite{Su17}, Lemma~4.1), one has:
$$b_0(C)=\frac1{12}\left(\frac1{f'(0)}-f'(0)\right)
$$
In this section, we will compute $b_{1/2}(C)$. Note that this is also the purpose of Proposition~3.5 in~\cite{Su17-2}, which does, however, not agree with the results that we obtain in Corollary~\ref{cor:b1total} and Theorem~\ref{thm:irrat} below. This will be in part, but not only, be a consequence of the difference between~\eqref{eq:Aprime0} and formula~(3.1) of~\cite{Su17-2}.

\begin{remark}
\label{rem:b1}
Let
$$L_0:=L_{b,0}=-\partial_r^2+r^{-2}A(0),
$$
the so-called ``frozen'' operator.  Its resolvent is
$$G_0:=G_{b,0}
$$
and can be written as
$$G_0(\zeta)=\bigoplus_{a\in\operatorname{spec}(A(0))}(-\partial_r^2 +r^{-2}a+\zeta^2)^{-1}\otimes\pi_a\,,
$$
where the sum is over all eigenvalues of $A(0)$, and the endomorphism $\pi_a$ of $L^2((0,\varepsilon)\times S^1)$ denotes fiberwise projection onto the
eigenspace~$E_a$ of the eigenvalue~$a$ of~$A(0)$; that is,
$$(\pi_a \phi)(r,\theta)=(\operatorname{proj}_{E_a}(\phi(r,\,.\,)))(\theta).
$$
{}From Lemma~4.3 of~\cite{BS87} it follows that
$${\partial_r}_{|r=0} G_{b,r}=-G_0 X^{-1/2} A'(0) X^{-1/2}G_0\,,
$$
where $X$ is the operator acting on functions on $(0,\varepsilon)\times S^1$ as multiplication by the first coordinate.
Since $A'(0)$ commutes with $L_0$ and, thus, with $G_0$, it follows that
\begin{equation}
\label{eq:firstderiv}
{\partial_r}_{|r=0}\operatorname{tr}_H G^m_{b,r}(1,1;\zeta)=-m\operatorname{tr}_H (G_0^{m+1}A'(0))(1,1;\zeta).
\end{equation}
Recall from~\eqref{eq:A0} that $A(0)\ge-\frac14$.
For each eigenvalue~$a$ of~$A(0)$ write
$$\nu(a):=\sqrt{a+\frac14}.
$$
By formula~(7.9) of~\cite{BS87}, the operator $(-\partial_r^2+r^{-2}a+\zeta^2)^{-m}$ has a kernel~$k^m_{\nu(a)}$ which is given on the diagonal by
$$k_{\nu(a)}^m(r,r;\zeta)=\frac1{(m-1)!}\left(-\frac1{2\zeta}\frac\partial{\partial\zeta}\right)^{m-1}r I_{\nu(a)}(r\zeta) K_{\nu(a)}(r\zeta),
$$
where $I_\nu$ and $K_\nu$ are the Bessel functions of the first and second kind, respectively.
In particular, 
$$k_{\nu(a)}^m(1,1;\zeta)=\frac1{(m-1)!}\left(-\frac1{2\zeta}\frac\partial{\partial\zeta}\right)^{m-1}I_{\nu(a)}(\zeta) K_{\nu(a)}(\zeta)
$$
and, by formula~(7.10) of~\cite{BS87},
$$\operatorname{tr}_H G_0^m(1,1;\zeta)=\sum_{a\in\operatorname{spec}(A(0))}
\frac1{(m-1)!}\left(-\frac1{2\zeta}\frac\partial{\partial\zeta}\right)^{m-1}I_{\nu(a)}(\zeta) K_{\nu(a)}(\zeta),
$$
where the eigenvalues of~$A(0)$ are repeated according to their multiplicity $\operatorname{dim}(E_a)$ from now on.
Note that on the image of~$\pi_a$, the operator $A'(0)=k_f A(0)$ acts as multiplication by the number $k_f a$, where $k_f=-f''(0)/f'(0)$ as in~\eqref{eq:Aprime0}.
Thus,
$$\operatorname{tr}_H (G_0^{m+1}A'(0))(1,1;\zeta)=\sum_{a\in\operatorname{spec}(A(0))}
\frac1{m!}\left(-\frac1{2\zeta}\frac\partial{\partial\zeta}\right)^{m}I_{\nu(a)}(\zeta) K_{\nu(a)}(\zeta) k_f a\,d\zeta.
$$
By~\eqref{eq:b-explicit} and~\eqref{eq:firstderiv}, this implies:
\begin{equation}
\label{eq:b1kf}
\begin{aligned}
b_{1,m}(C)&=\fint\zeta^{2m}\,{\partial_r}_{|r=0}\operatorname{tr}_H G_{b,r}^m(1,1;\zeta)\,d\zeta\\
&=-m\fint_0^\infty\zeta^{2m}\,\operatorname{tr}_H (G_0^{m+1}A'(0))(1,1;\zeta)\,d\zeta\\
&=-\frac{k_f}{(m-1)!}\fint_0^\infty\zeta^{2m}\left(-\frac1{2\zeta}\frac\partial{\partial\zeta}\right)^{m}\sum_{a\in\operatorname{spec}(A(0))}a\,I_{\nu(a)}(\zeta) K_{\nu(a)}(\zeta)\,d\zeta
\end{aligned}
\end{equation}
\end{remark}

\begin{lemma}
\label{lem:b1res}
\begin{equation*}
b_{1,m}(C)=-\frac{k_f}{(m-1)!}{\operatorname{Res}_0}_{|s=0}\left[\frac1{4\sqrt\pi}\Gamma\left(m+\frac12+\frac s2\right)\Gamma\left(-\frac s2\right)\sum_{a\in\operatorname{spec}(A(0))}
a\,\frac{\Gamma\left(\nu(a)+\frac12+\frac s2\right)}{\Gamma\left(\nu(a)+\frac12-\frac s2\right)}\right],
\end{equation*}
where the elements $a\in\operatorname{spec}(A(0))$ are counted with multiplicity.
\end{lemma}

\begin{proof}
By formula~(7.11) of~\cite{BS87} (substituting $m+1$ for~$m$ in that formula), we have
\begin{equation*}
\int_0^\infty \zeta^w\left(-\frac1{2\zeta}\frac\partial{\partial\zeta}\right)^m I_\nu(\zeta)K_\nu(\zeta)\,d\zeta=
\frac1{4\sqrt\pi}\,\Gamma\left(\frac{w+1}2\right)\Gamma\left(m-\frac w2\right)\frac{\Gamma(\nu+\frac 12+\frac w2 -m)}{\Gamma(\nu +\frac 12 -\frac w2+m)}
\end{equation*}
For $w=2m+s$, this gives:
$$\int_0^\infty \zeta^{2m+s}\left(-\frac1{2\zeta}\frac\partial{\partial\zeta}\right)^m I_\nu(\zeta)K_\nu(\zeta)\,d\zeta=
\frac1{4\sqrt\pi}\,\Gamma\left(m+\frac 12+\frac s2\right)\Gamma\left(-\frac s2\right)\frac{\Gamma(\nu+\frac 12+\frac s2)}{\Gamma(\nu +\frac 12 -\frac s2)}
$$
Together with~\eqref{eq:b1kf} and \eqref{eq:regint}, this implies the statement. 
\end{proof}

\begin{remark}
\label{rem:polespec}

(i)
In view of the factor $\Gamma(-\frac s2)$ in the formula from Lemma~\ref{lem:b1res}, the fact that $\Gamma$ has a pole at $s=0$ plays an important role for the subsequent computation of $b_{1,m}(C)$. Note that this is a very different situation compared to
formula~(7.12) of~\cite{BS87}, which gives a similar expression
concerning (in our notation) $b_{0,m}(C)$. As it is, the pole of $\Gamma$ at $s=0$ seems to have been ignored in the computation starting in the lower part of p.~9 in~\cite{Su17-2}. In that paper, $b_{1,m}(C)$ is claimed to be equal to $\frac{f''(0)}{f'(0)}\cdot\frac{5\Gamma(m+\frac12)}{96\sqrt\pi(m-1)!}$ (see p.~9 of~\cite{Su17-2}), that is, $-k_f\cdot\frac{5\Gamma(m+\frac12)}{96\sqrt\pi(m-1)!}$ (in our notation). This would mean that $b_{1/2}(C)$ were equal to $\frac{f''(0)}{f'(0)}\cdot\frac5{96\sqrt\pi}$. Note that this expression is constant under constant rescalings of~$f$ and contradicts our formula for $b_{1,m}(C)$ in Corollary~\ref{cor:b1total} below, as well as its irrational dependence on~$f'(0)$ described in Section~\ref{sec:irrat}.

(ii)
The numbers $\nu(a)=\sqrt{a+\frac14}$, when $a$ runs through $\operatorname{spec}(A(0))$ (with multiplicities) constitute the series
$$0, \alpha, \alpha, 2\alpha, 2\alpha, 3\alpha, 3\alpha, 4\alpha, 4\alpha, \dots
$$
This follows immediately from~\eqref{eq:A0}, recalling that $\alpha=\frac1{f'(0)}$.

(iii)
In the following we will use the Beta function
$$B(z,w)=\frac{\Gamma(z)\Gamma(w)}{\Gamma(z+w)}=\int_0^1 (1-t)^{z-1} t^{w-1}\,dt,
$$
considered as a meromorphic function in~$z\in\mathbb{C}$ for each fixed $w\in\mathbb{C}\setminus\{0,-1,-2,\dots\}$, and vice versa.
The integral on the right side converges only for $\operatorname{Re}(z)>0$ and $\operatorname{Re}(w)>0$, but can, in other cases, be interpreted using the above formula.

(iv)
The functional equation $\Gamma(z+1)=z\Gamma(z)$ implies functional equations for the Beta function, for example:
$$B(z,w)=\frac{z+w}{w} B(z,w+1) \mbox{ and } B(z+1,w-1)=\frac z{w-1}B(z,w)
$$
\end{remark}

\begin{corollary}
\label{cor:b1resshort}
Using $a=\nu(a)^2-\frac14$, we conclude from Lemma~\ref{lem:b1res} and Remark~\ref{rem:polespec}(ii):
\begin{equation*}
b_{1,m}(C)=-\frac{k_f}{4\sqrt{\pi}(m-1)!}{\operatorname{Res}_0}_{|s=0}\left[\psi_m(s)\cdot h_\alpha(s)\right],
\end{equation*}
where $\psi_m$ is the meromorphic function given by
$$\psi_m(s):=\frac{\Gamma(m+\frac12+\tfrac s2)\Gamma(-\frac s2)}{\Gamma(-s)},
$$
and
\begin{equation*}
\begin{aligned}
h_\alpha(s):=&{}-\frac14\cdot\frac{\Gamma(\frac12+\frac s2)\Gamma(-s)}{\Gamma(\frac12 - \frac s2)}
     +2\sum_{n=1}^\infty \left(\alpha^2 n^2-\frac14\right)\frac{\Gamma(n\alpha+\frac12+\frac s2)\Gamma(-s)}{\Gamma(n\alpha+\frac12-\frac s2)}\\
\end{aligned}
\end{equation*}
\end{corollary}

\begin{remark}
\label{rem:psismooth}
Note that $\psi_m$ has no pole at $s=0$. Thus, it is holomorphic in some open neighborhood of $s=0$. Its value at $s=0$ is
$$\psi_m(0)=2\,\Gamma(m+\tfrac12)
$$
The function~$h_\alpha$ is meromorphic in some neighborhood of $s=0$. We are going to show that~$h_\alpha$ does actually not have a pole in $s=0$, either; see Corollary~\ref{cor:b1total}.
\end{remark}

\begin{definition}
\label{def:hka}
For each $\alpha>0$, $k\in\mathbb{N}_0$, and $z$ in
\begin{equation}
\label{eq:W}
W:=\{z\in\mathbb{C}: |z|<1 \mbox{\ and\ }|1-z|<1\}
\end{equation}
we write
\begin{equation*}
h_{k,\alpha}(z):=\sum_{n=0}^\infty \alpha^k n^k (1-z)^{n\alpha},
\end{equation*}
where $(1-z)^{n\alpha}:=\exp(n\alpha\log(1-z))$; here, $\log$ denotes the main branch of the complex logarithm.
\end{definition}

Note that for $n=0$, the summand $\alpha^k n^k (1-z)^{n\alpha}$ of $h_{k,\alpha}(z)$ equals $1$ for $k=0$ (and $0$ for $k>0$). Taking this into account, we obtain (noting that the term $+\frac14$ in the bracket of the following equation serves for achieving the correct coefficient $-\frac14$ in the first term of $h_\alpha(s)$ from Corollary~\ref{cor:b1resshort}):

\begin{corollary}
\label{cor:b1reshka}
For the function~$h_\alpha$ from Corollary~\ref{cor:b1resshort} we get, using the monotone convergence theorem for the integral over $(0,1)$,
\begin{equation*}
h_\alpha(s)= \int_0^1\left(2h_{2,\alpha}(t)-\frac12 h_{0,\alpha}(t)+\frac14\right)(1-t)^{\frac s2-\frac 12} t^{-s-1}\,dt.
\end{equation*}
\end{corollary}

\begin{notrems}
\label{notrems:hka}
Let $\alpha>0$.

(i)
For each $z\in W$ (which was defined in~\eqref{eq:W}) we have
\begin{equation*}
\begin{aligned}
h_{0,\alpha}(z)&=\sum_{n=0}^\infty ((1-z)^\alpha)^n=\frac 1{1-(1-z)^\alpha}=
\frac1{\alpha z-\binom\alpha 2 z^2+\binom\alpha 3 z^3-\ldots}\\
&=\frac1{\alpha z}\cdot \frac1{1-Q_\alpha(z)}=\frac1{\alpha z}+\frac12\left(1-\frac1\alpha\right)+O(z),
\end{aligned}
\end{equation*}
where
$$Q_\alpha(z)=\frac12(\alpha-1)z-\frac16(\alpha-1)(\alpha-2)z^2+\ldots
$$
is a power series vanishing in $z=0$ and converging for each $|z|<1$.
Thus, $h_{0,\alpha}$ has a meromorphic extension, denoted $h_{0,\alpha}$ again, to the open unit disc
and has a simple pole at $z=0$ with residue~$\frac1\alpha$.
In particular, the function
$$
z\mapsto h_{0,\alpha}(z)-\frac 1{\alpha z}=h_{0,\alpha}(z)-\frac1\alpha h_{0,1}(z)
$$
is holomorphic in some open neighborhood of $z=0$.

(ii)
More precisely, the meromorphic function $h_{0,\alpha}$ on the open unit disc has a pole in~$z$ if and only if $|1-z|=1$ and
$$\log(1-z)\in i\cdot\left(\tfrac{2\pi}\alpha\mathbb{Z}\,\cap\,(-\tfrac\pi3,\tfrac\pi3)\,\right)
$$
(in fact, if $|1-z|=1$ then the condition $|z|<1$ is equivalent to $\log(1-z)\in i\,(-\frac\pi3,\frac\pi3)$).
For $\alpha\le6$, this is satsfied only for~$z=0$. For $\alpha>6$, there are additional poles in the open unit disc. Among these, the ones with the smallest nonzero distance
to $z=0$ are the points $1-\exp(\pm\frac{2\pi i}\alpha)$ with norm
 $2\left|\sin(\frac\pi\alpha)\right|$. Thus, the function $z\mapsto h_{0,\alpha}(z) - \frac1{\alpha z}$ is holomorphic on the open disc
\begin{equation}
\label{eq:valpha}
V_\alpha:=\{z\in\mathbb{C}:|z|<r_\alpha\},\text{ where }r_\alpha:=\begin{cases}1 & \text{ if }0<\alpha<6,\\
2\left|\sin\tfrac\pi\alpha\right| & \text{ if }\alpha\ge6,
\end{cases}
\end{equation}
and the above function is still holomorphic on
\begin{equation}
\label{eq:walpha}
W_\alpha:=W\cup V_\alpha.
\end{equation}

(iii)
For each $k\in\mathbb{N}_0$ and $z\in W$, we have
$$\frac d{dz} h_{k,\alpha}(z)=\sum_{n=1}^\infty\alpha^k n^k\cdot n\alpha(1-z)^{n\alpha-1}\cdot(-1)
= -\frac1{1-z} h_{k+1,\alpha}(z).
$$
Hence,
\begin{equation}
\label{eq:hka-rec}
h_{k,\alpha}(z)=\left(-(1-z)\frac d{dz}\right)^k h_{0,\alpha}(z).
\end{equation}
Consequently, (i) implies that for {\it each\/} $k\in\mathbb{N}_0$, the function $h_{k,\alpha}$ has a meromorphic extension
to the open unit disc, with the same set of poles as~$h_{0,\alpha}$. Moreover, since $h_{0,\alpha}(z)-\frac1{\alpha z}$ is holomorphic near $z=0$, it follows that
for each $k\in\mathbb{N}_0$, the singular part $h_{k,\alpha}^{\textrm{sing}}(z)$ of the Laurent series of $h_{k,\alpha}(z)$ at $z=0$ is given by
\begin{equation}
\label{eq:hka-singrec}
\begin{aligned}
h_{k,\alpha}^{\textrm{sing}}(z)=\frac1\alpha\left(-(1-z)\frac d{dz}\right)^k\left(\frac1z\right)
=\frac1\alpha h_{k,1}(z)=\frac1\alpha h_{k,1}^{\textrm{sing}}(z)
\end{aligned}
\end{equation}

(iv)
We denote the regular part of $h_{k,\alpha}(z)$ with respect to the pole at $z=0$ by
$$h_{k,\alpha}^{\textrm{reg}}(z):=h_{k,\alpha}(z)-h_{k,\alpha}^{\textrm{sing}}(z)=h_{k,\alpha}(z)-\frac1\alpha h_{k,1}(z).
$$
This is again a holomorphic function on the open set~$W_\alpha$ from~\eqref{eq:walpha}. Equation~\eqref{eq:hka-rec} implies that $h_{k,\alpha}^{\textrm{reg}}$ satisfies the analogous equation
\begin{equation}
\label{eq:hka-regrec}
h_{k,\alpha}^{\textrm{reg}}(z)=\left(-(1-z)\frac d{dz}\right)^k h_{0,\alpha}^{\textrm{reg}}(z).
\end{equation}

(v)
For future use, we also introduce the following functions:
$$h_{k,\alpha}^{\textrm{pos}}(z):=h_{k,\alpha}^{\textrm{reg}}(z)-h_{k,\alpha}^{\textrm{reg}}(0)
$$
and
$$\hat h_{k,\alpha}(z):=\frac{h_{k,\alpha}^{\textrm{pos}}(z)}z
$$
\end{notrems}
These are again holomorphic on the open set~$W_\alpha$ from~\eqref{eq:walpha}.

(vi)
Note that $h_{0,\alpha}(z)$ for general $\alpha>0$ is not defined on any open neighborhood of $z=1$. However, it
extends continuously (with value $h_{0,\alpha}(1)=1$) to the limit point $z=1$ of the open set $W\subset W_\alpha$. 
Using equation~\eqref{eq:hka-rec}, one easily sees that {\it each\/} of the functions $h_{k,\alpha}$ extends continuously to $z=1$ (with value
$h_{k,\alpha}(1)=0$ for $k>0$). Obviously, continuous extendability in $z=1$ now follows for each of the functions $h_{k,\alpha}^{\textrm{reg}}$ and $\hat h_{k,\alpha}$, too.
In particular, for each $k\in\mathbb{N}_0$ and $\alpha>0$, we obtain a continuous function
$$\hat h_{k,\alpha}:[0,1]\to\mathbb{R}
$$
on the closed unit interval (which coincides with the restriction of the holomorhpic function $\hat h_{k,\alpha}:W_\alpha\to\mathbb{C}$ on the half-open
interval $[0,1)$\,).

\begin{lemma}
\label{lem:hka-sing}
For each $\alpha>0$ we have:
\begin{itemize}
\item[(i)]
$$\int_0^1 h_{0,\alpha}^{\textrm{sing}}(t)(1-t)^{\frac s2-\frac12}t^{-s-1}\,dt
=\frac 1{2\alpha}\int_0^1 (1-t)^{\frac s2 - \frac12}t^{-s-1}\,dt.
$$

\item[(ii)]
The function 
$$s\mapsto \int_0^1 h_{2,\alpha}^{\textrm{sing}}(t)(1-t)^{\frac s2-\frac12}t^{-s-1}\,dt
$$
vanishes identically in $s\in\mathbb{C}$.
\end{itemize}
\end{lemma}

\begin{proof}

(i)
By Remark~\ref{notrems:hka}(i),
\begin{equation}
\label{eq:h0sing}
h_{0,\alpha}^{\textrm{sing}}(z)=\frac1{\alpha z}.
\end{equation}
Thus, the left hand side of the statement equals
$$\alpha^{-1}B\left(\frac s2+\frac12,-s-1\right)
=\alpha^{-1}\cdot\frac{\frac s2+\frac12-s-1}{-s-1}B\left(\frac s2+\frac12,-s\right)=\frac1{2\alpha}B\left(\frac s2+\frac12,-s\right),
$$
where we have used Remark~\ref{rem:polespec}(iv) in the first equation.

(ii)
By equation~\eqref{eq:hka-singrec},
\begin{equation*}
\alpha h_{2,\alpha}^{\textrm{sing}}(z)=\left(-(1-z)\frac d{dz}\right)\left((1-z)\frac1{z^2}\right)=(1-z)^2\frac2{z^3}+(1-z)\frac1{z^2}
\end{equation*}
Thus, 
\begin{equation*}
\begin{aligned}
\int_0^1 \alpha h_{2,\alpha}^{\textrm{sing}}(t)(1-t)^{\frac s2-\frac12}t^{-s-1}\,dt
&=2B\left(\frac s2+\frac 52, -s-3\right)+B\left(\frac s2+\frac 32,-s-2\right)
\\&=\left[2\frac{\frac s2+\frac 32}{-s-3}+1\right]B\left(\frac s2+\frac 32,-s-2\right)
=0,
\end{aligned}
\end{equation*}
where Remark~\ref{rem:polespec}(iv) is used in the last equation.
\end{proof}

\begin{definition}
\label{def:powera}

(i)
For each $\alpha>0$ we define a complex power series $P_\alpha(w)$ by
$$P_\alpha(w):=\sum_{j=0}^\infty\frac1{(j+1)!}B_{j+1}\alpha^{j+1} w^j,
$$
where the $B_j$ denote the Bernoulli numbers. The radius of convergence of~$P_\alpha$ is
$$R_\alpha:=\frac{2\pi}\alpha\,.
$$
We write
$$U_\alpha:=\{z\in\mathbb{C}: |z|<1 \text{ and } |\log(1-z)|<R_\alpha\}.
$$

(ii)
We define a families of holomorphic functions $\Phi_k$ and~$\hat\Phi_k$ ($k\in\mathbb{N}_0$) on the open unit disc $\{z\in\mathbb{C}:|z|<1\}$ by
\begin{equation*}
\begin{aligned}
\Phi_k(z)&:=\left(-(1-z)\frac d{dz}\right)^k\left(\frac 1z+\frac 1{\log(1-z)}\right),\\
\hat\Phi_k(z)&:=\frac{\Phi_k(z)-\Phi_k(0)}z\,.
\end{aligned}
\end{equation*}

\end{definition}

\begin{remark}
\label{rem:Pk-phicont}

(i)
The $k$th derivative $P^{(k)}$ of $P_\alpha$ satisfies
\begin{equation}
\label{eq:Pka}
P^{(k)}_\alpha(w)=\sum_{j=0}^\infty\frac1{j!\cdot(j+k+1)}B_{j+k+1}\alpha^{j+k+1}w^j
\end{equation}
for all $w\in\mathbb{C}$ with $|w|<R_\alpha$\,.

(ii)
The functions $\Phi_k$ are indeed holomorphic on the open unit disc because the singularity of~$\Phi_0$ at $z=0$ is removable (with value $\Phi_0(0)=\frac12$). Consequently, the functions $\hat\Phi_k$, too, are holomorphic on the open unit disc.
Moreover, we note that the functions~$\Phi_k$ continuously extend to the limit point $z=1$ (with value~$1$ for $k=0$ and value~$0$ for~$k>0$). Consequently, the functions~$\hat\Phi_k$, too,  continuously extend to the point $z=1$.
In particular, $\Phi_k$ and~$\hat\Phi_k$ induce continuous functions $\Phi_k:[0,1]\to\mathbb{R}$ and $\hat\Phi_k:[0,1]\to\mathbb{R}$, respectively.
\end{remark}

\begin{lemma}
\label{lem:hka-taylor}
Using the notation of~\ref{notrems:hka},
we have the following for each $\alpha>0$:
\begin{itemize}
\item[(i)]
For all $z\in U_\alpha$\,,
\begin{equation*}
\begin{aligned}
h_{0,\alpha}(z)&=-\frac1{\alpha\log(1-z)}-\frac1{\alpha}P_\alpha(\log(1-z))\text{ and }\\
h_{0,\alpha}^{\textrm{reg}}(z)&=-\frac1{\alpha}\Phi_0(z)-\frac1\alpha P_\alpha(\log(1-z)).
\end{aligned}
\end{equation*}
Moreover, for all $z\in U_\alpha\cap U_1$\,,
$$h_{0,\alpha}^{\textrm{reg}}(z)=\frac1\alpha(P_1-P_\alpha)(\log(1-z)).
$$

\item[(ii)]
For each $k\in\mathbb{N}_0$ and all $z\in U_\alpha$\,,
\begin{equation}
\label{eq:hka-reg}
h_{k,\alpha}^{\textrm{reg}}(z)=-\frac1\alpha\Phi_k(z)-\frac1\alpha P_\alpha^{(k)}(\log(1-z)).
\end{equation}
Moreover, for all $z\in U_\alpha\cap U_1$\,,
\begin{equation*}
\begin{aligned}
h_{k,\alpha}^{\textrm{reg}}(z)&=\frac1\alpha(P^{(k)}_1-P^{(k)}_\alpha)(\log(1-z))\\
&=\frac1\alpha\sum_{j=0}^\infty\frac1{j!\cdot(j+k+1)}B_{j+k+1}(1-\alpha^{j+k+1})(\log(1-z))^j.
\end{aligned}
\end{equation*}
In particular,
\begin{equation}
\label{eq:hka-reg-0}
h_{k,\alpha}^{\textrm{reg}}(0)=\frac1\alpha\cdot\frac1{k+1}B_{k+1}(1-\alpha^{k+1}).
\end{equation}

\item[(iii)]
For each $k\in\mathbb{N}_0$ and all $z\in U_\alpha$\,,
\begin{equation}
\label{eq:hka-hat}
\begin{aligned}
\hat h_{k,\alpha}(z)&=-\frac1\alpha\hat\Phi_k(z)-\frac1{\alpha z} (P_\alpha^{(k)}(\log(1-z))-P_\alpha^{(k)}(0))\\
&=-\frac1\alpha\hat\Phi_k(z)-\frac1{\alpha z}\sum_{j=1}^\infty\frac1{j!\cdot(j+k+1)}B_{j+k+1}\alpha^{j+k+1}(\log(1-z)^j).
\end{aligned}
\end{equation}
\end{itemize}
\end{lemma}

\begin{proof}
(i)
For all $z\in U_\alpha$ we have
\begin{equation*}
\begin{aligned}
h_{0,\alpha}(z)&=\frac1{1-(1-z)^\alpha}=\frac1{1-\exp(\alpha\log(1-z))}=-\frac{\alpha\log(1-z)}{\exp(\alpha\log(1-z))-1}\cdot\frac1{\alpha\log(1-z)}\\
&=-\sum_{\ell=0}^\infty\frac1{\ell!}B_\ell(\alpha\log(1-z))^{\ell}\cdot\frac1{\alpha\log(1-z)}=-\frac1{\alpha\log(1-z)}-\frac1\alpha P_\alpha(\log(1-z)).
\end{aligned}
\end{equation*}
This shows the first statement. The second statement follows by~\eqref{eq:h0sing}. In order to conclude the third statement, we let $\alpha=1$ in the first statement, which gives
$$P_1(\log(1-z))=-h_{0,1}(z)-\frac1{\log(1-z)}=-\frac1z-\frac1{\log(1-z)}=-\Phi_0(z)
$$
for all $z\in U_1$.

(ii)
This follows from (i) using~\eqref{eq:hka-regrec}, \eqref{eq:Pka}, and
$$\left(-(1-z)\frac d{dz}\right)(\log(1-z)^j)=j(\log(1-z))^{j-1}
$$
for all $j\in\mathbb{Z}$.

(iii)
This is an immediate consequence of~\eqref{eq:hka-reg}, \eqref{eq:Pka}, and $\log(1)=0$.
\end{proof}

\begin{remark}
Note that the open set $U_\alpha$ on which $P_\alpha(\log(1-z))$ converges does, in general, not contain the entire open disc~$V_\alpha$ from~\eqref{eq:valpha} on which the functions $h_{k,\alpha}^{\textrm{reg}}$ are holomorphic. The reason is that the formulas from Lemma~\ref{lem:hka-taylor} do not directly express $h_{k,\alpha}^{\textrm{reg}}(z)$ as a power series in~$z$. Doing this would rearrange the summation and make the resulting power series converge on~$V_\alpha$\,.
\end{remark}

\begin{proposition}
\label{prop:hpos}
The function $h_\alpha$ from Corollary~\ref{cor:b1resshort} satisfies
\begin{equation*}
\begin{aligned}
h_\alpha(s)&=\int_0^1\left(2h_{2,\alpha}^{\textrm{pos}}(t)-\frac12 h_{0,\alpha}^{\textrm{pos}}(t)\right)(1-t)^{\frac s2-\frac12} t^{-s-1}\,dt\\
&=\int_0^1\left(2\hat h_{2,\alpha}(t)-\frac12 \hat h_{0,\alpha}(t)\right)(1-t)^{\frac s2-\frac12} t^{-s}\,dt,
\end{aligned}
\end{equation*}
where $h_{k,\alpha}^{\textrm{pos}}$ and $\hat h_{k,\alpha}$ are defined as in~\ref{notrems:hka}(v).
\end{proposition}

\begin{proof}
Note that the second equation is immediate, so it remains to show the first equation.
Recall from Corollary~\ref{cor:b1reshka} that
$$h_\alpha(s)= \int_0^1\left(2h_{2,\alpha}(t)-\frac12 h_{0,\alpha}(t)+\frac14\right)(1-t)^{\frac s2-\frac 12} t^{-s-1}\,dt.
$$
We already know from Lemma~\ref{lem:hka-sing}(ii) that we can replace $h_{2,\alpha}$ by $h_{2,\alpha}^{\textrm{sing}}$ in this formula. Moreover,
by~\eqref{eq:hka-reg-0} and $B_{0+2+1}=B_3=0$,
\begin{equation*}
h_{2,\alpha}^{\textrm{reg}}(0)=0,
\end{equation*}
so we can actually replace $h_{2,\alpha}$ by $h_{2,\alpha}^{\textrm{pos}}$ in the above formula. It remains to show that
\begin{equation}
\label{eq:0sing0}
\int_0^1\left(-\frac 12 h_{0,\alpha}^{\textrm{sing}}(t)-\frac 12 h_{0,\alpha}^{\textrm{reg}}(0)+\frac14\right)(1-t)^{\frac s2-\frac12} t^{-s-1}\,dt=0.
\end{equation}
By Remark~\ref{notrems:hka}(i) (or~\eqref{eq:hka-reg-0} and $B_1=-\frac12$) we have
\begin{equation*}
h_{0,\alpha}^{\textrm{reg}}(0)=-\frac12\left(\frac1\alpha-1\right),
\end{equation*}
giving $-\frac12 h_{0,\alpha}^{\textrm{reg}}(0)+\frac14=\frac1{4\alpha}$.
Now~\eqref{eq:0sing0} follows from Lemma~\ref{lem:hka-sing}(i).
\end{proof}

Using continuity of the functions $\hat h_{k,\alpha}$ on the closed interval $[0,1]$ (recall~\ref{notrems:hka}(vi)), we now obtain:

\begin{corollary}
\label{cor:b1total}
The function $h_\alpha$ from Corollary~\ref{cor:b1resshort} has a finite value at $s=0$. In particular, that Corollary together with Remark~\ref{rem:psismooth} and
Proposition~\ref{prop:hpos} implies
\begin{equation*}
\begin{aligned}
b_{1,m}(C)&=-\frac{k_f}{4\sqrt\pi (m-1)!}\psi_m(0)h_\alpha(0)
=-\frac{k_f \cdot\Gamma(m+\frac12)}{\sqrt\pi (m-1)!}\int_0^1 \left(\hat h_{2,\alpha}(t)-\frac14\hat h_{0,\alpha}(t)\right)(1-t)^{-\frac12}\,dt,\\
&=-\frac{2 k_f \cdot\Gamma(m+\frac12)}{\sqrt\pi (m-1)!}\int_0^1 \left(\hat h_{2,\alpha}(1-u^2)-\frac14\hat h_{0,\alpha}(1-u^2)\right)\,du,
\end{aligned}
\end{equation*}
where, as introduced before, $k_f=-f''(0)/f'(0)$, $\alpha=1/f'(0)$, and the functions $\hat h_{k,\alpha}$ are defined as in~\ref{notrems:hka}(v).
Using~\eqref{eq:bhalf}, we finally conclude
\begin{equation}
\label{eq:bhalftotal}
b_{1/2}(C)
=-\frac{2 k_f}{\sqrt\pi}\int_0^1\left(\hat h_{2,\alpha}(1-u^2)-\frac14\hat h_{0,\alpha}(1-u^2)\right)\,du
\end{equation}
\end{corollary}

\Section{Irrational dependence of $b_{1/2}(C)$ on constant rescalings of the distance circles}
\label{sec:irrat}

\noindent
In this section, we will show that if $f''(0)\ne0$ then, under rescalings $\lambda f$ of~$f$ by a constant $\lambda>0$, the coefficient $b_{1/2}(C)$ is {\it not\/} a rational function of~$\lambda$. Since, by~\ref{notrems:Ccoord}, the length of the distance circles is $2\pi f(r)$ for small $r>0$, this means
that $b_{1/2}(C)$ does not depend rationally on the scaling factor for constant rescalings of small distance circles aronnd~$C$ if $f''(0)\ne0$.
Note that the factor $k_f=-f''(0)/f'(0)$ in~\eqref{eq:bhalftotal} is invariant under such rescalings, while $f'(0)$ changes linearly. Recall that $\alpha$ is just the inverse of~$f'(0)$.
So, what we intend to show is that the integral in~\eqref{eq:bhalftotal} -- or, equivalently, $\alpha$ times that integral -- is not a rational function of the parameter~$\alpha$. The aim of this section is, thus, to prove the following theorem:

\begin{theorem}
\label{thm:irrat}
\begin{equation}
\label{eq:defF}
F:(0,\infty)\ni\alpha\mapsto\int_0^1\left(\alpha\hat h_{2,\alpha}(1-u^2)-\frac14\alpha\hat h_{0,\alpha}(1-u^2)\right)\,du\in\mathbb{R}
\end{equation}
is not a rational function of~$\alpha$. In particular, if $f''(0)\ne0$ then $b_{1/2}(C)$ does not change rationally under constant rescalings of the distance circles near the singularity~$C$.
\end{theorem}

We first give several preparations for the proof.

\begin{definition}
\label{def:calpha}
For each $\alpha>0$ let
\begin{equation*}
c_\alpha:=e^{-\frac{\pi}{2\alpha}}.
\end{equation*}
\end{definition}

\begin{lemma}
\label{lem:intsmall}
For each $k\in\mathbb{N}_0$ there exists a constant~$\Lambda_k>0$, independent of~$\alpha$, such that for each $0<\alpha\le1$,
$$\max\left\{|\alpha\hat h_{k,\alpha}(1-u^2)|:u\in[0,c_\alpha]\right\}\le\Lambda_k\,.
$$
\end{lemma}

\begin{proof}
The definitions and formulas of~\ref{notrems:hka} together with~\eqref{eq:hka-reg-0} imply that
\begin{equation*}
\begin{aligned}
\alpha\hat h_{k,\alpha}(z)&=\frac1z\left(\alpha\left(-(1-z)\frac d{dz}\right)^k h_{0,\alpha}^{\textrm{reg}}(z)-\alpha h_{k,\alpha}^{\textrm{reg}}(0)\right)\\
&=\frac1z\left(\left(-(1-z)\frac d{dz}\right)^k \left(\frac\alpha{1-(1-z)^\alpha}-\frac1z\right)-\frac1{k+1}B_{k+1}(1-\alpha^{k+1})\right)
\end{aligned}
\end{equation*}
This implies that for each fixed $k\in\mathbb{N}_0$\,, $\alpha\hat h_{k,\alpha}(z)$ is a linear combination of finite products of the terms
$$\alpha, \ \frac1z, \ 1-z, \ \frac1{1-(1-z)^\alpha}, \ (1-z)^\alpha.
$$
Consequently, $\alpha\hat h_{k,\alpha}(1-u^2)$ is a linear combination of finite products of the terms
$$\alpha, \ \frac1{1-u^2},  \ u^2, \ \frac1{1-u^{2\alpha}}, \ u^{2\alpha}.
$$
For each $0<\alpha\le1$ one has $c_\alpha\le e^{-\frac\pi 2}<\frac12$ and, for every $u\in[0,c_\alpha]$,
$$0\le u^2\le u^{2\alpha}\le e^{-\pi}<\frac12\text{ \ and \ }0<\frac1{1-u^2}\le\frac1{1-u^{2\alpha}}\le\frac1{1-e^{-\pi}}<2.
$$
In view of the structure of $\alpha\hat h_{k,\alpha}(1-u^2)$ described above, these estimates obviously imply the statement.
\end{proof}

\begin{proposition}
\label{prop:Fapprox}
For each $j\in\mathbb{N}$ write
$$\tau_j:(0,1)\ni u\mapsto\frac{(\log(u^2))^j}{1-u^2}\in\mathbb{R}\text{\ \ and\ \ }I_j:=\int_0^1\tau_j(u)\,du\in\mathbb{R}.
$$
Then, for all $k\in\mathbb{N}_0$ and $r\in\mathbb{N}$ we have
\begin{equation}
\label{eq:hkaint-approx}
\int_0^1\alpha\hat h_{k,\alpha}(1-u^2)\,du
=-\int_0^1\hat\Phi_k(1-u^2)\,du
-\sum_{j=1}^r\frac{I_j}{j!\cdot(j+k+1)}B_{j+k+1}\alpha^{j+k+1}
+o(\alpha^{r+k+1})
\end{equation}
as $\alpha\searrow0$, where $\Phi_k$ is as in Definition~\ref{def:powera}(ii). In particular, for each $r\in\mathbb{N}$, the function~$F$ from~\eqref{eq:defF} satisfies
\begin{equation*}
\begin{aligned}
F(\alpha)
={}&-\int_0^1\left(\hat\Phi_2(1-u^2)-\frac14\hat\Phi_0(1-u^2)\right)\,du\\
&{}+\frac{I_1}{48}\alpha^2
-\sum_{j=1}^r \left(\frac{I_j}{j!\cdot(j+3)}-\frac14\cdot\frac{I_{j+2}}{(j+3)!}\right)B_{j+3}\alpha^{j+3}+o(\alpha^{r+3})
\end{aligned}
\end{equation*}
as $\alpha\to0$.
\end{proposition}

\begin{proof}
First of all, note that the numbers $I_j$ are indeed finite: The integrand $\tau_j$ of~$I_j$ extends continuously to $u=1$ (with value $-\log'(1)=-1$ for $j=1$ and value~$0$ for $j>1$); moreover, powers of $\log$ are integrable over the unit interval.

By $B_2=\frac16$ and the definition of~$F$, the first statement of the lemma immediately implies the second.
It remains to prove the first statement.
Fix $k\in\mathbb{N}_0$. For any $\alpha>0$ and $r\in\mathbb{N}$ let
\begin{equation*}
\begin{aligned}
D_*(\alpha,r)&:=\int_0^{c_\alpha} (\alpha\hat h_{k,\alpha}(1-u^2)+\hat\Phi_k(1-u^2))\,du+\sum_{j=1}^r\frac{\int_0^{c_\alpha}\tau_j(u)\,du}{j!\cdot(j+k+1)}B_{j+k+1}\alpha^{j+k+1},\\
D^*(\alpha,r)&:=\int_{c_\alpha}^1 (\alpha\hat h_{k,\alpha}(1-u^2)+\hat\Phi_k(1-u^2))\,du+\sum_{j=1}^r\frac{\int_{c_\alpha}^1\tau_j(u)\,du}{j!\cdot(j+k+1)}B_{j+k+1}\alpha^{j+k+1},
\end{aligned}
\end{equation*}
where $c_\alpha=e^{-\frac\pi{2\alpha}}$ as in Definition~\ref{def:calpha}. For the rest of this proof, we may assume that $\alpha<1$; in particular, $c_\alpha<\frac12$\,.
Our aim is to show that $D_*(\alpha,r)+D^*(\alpha,r)\in o(\alpha^{r+k+1})$ as $\alpha\searrow0$.

We will first show that $D_*(\alpha,r)\in o(\alpha^\infty)$, meaning that it is in $o(\alpha^n)$ for each $n\in\mathbb{N}$, as $\alpha\searrow0$.
From Lemma~\ref{lem:intsmall} and Remark~\ref{rem:Pk-phicont}(ii) we easily conclude that the first integral in $D_*(\alpha,r)$ is indeed in $o(\alpha^\infty)$. Moreover, for each $j\in\mathbb{N}$ one has $\int\log^j(v)\,dv=(-1)^j j!\cdot v\sum_{s=0}^j\frac1{s!}(-\log(v))^s$. Hence,
recalling that $\log(c_\alpha)=-\frac\pi{2\alpha}$ and $c_\alpha\in o(\alpha^\infty)$, we have:
$$\left|\int_0^{c_\alpha}\tau_j(u)\,du\right|\le\frac1{1-c_\alpha^2}\cdot\left|\int_0^{c_\alpha}\log^j(u^2)\,du\right|
<\frac43\cdot 2^j j!\cdot c_\alpha\sum_{s=0}^j\frac1{s!}(\tfrac\pi2)^s\alpha^{-s}\in o(\alpha^\infty)
$$
for each $j\in\mathbb{N}$.
This implies $D_*(\alpha,r)\in o(\alpha^\infty)$.

In order to prove~\eqref{eq:hkaint-approx}, it remains to show that $D^*(\alpha,r)\in o(\alpha^{r+k+1})$. Note that for $u\in[c_\alpha\,,1]$ we have $|\log(u^2)|\le\frac\pi\alpha$\,.
This implies that the compact set $\{1-u^2: u\in[c_\alpha\,,1]\}$ is contained in~$U_\alpha$\,; recall Definition~\ref{def:powera}.
In particular, the power series $P_\alpha^{(k)}(w)$
converges uniformly on the compact set $\{\log(u^2):u\in[c_\alpha\,,1]\}$.
Using~\eqref{eq:Pka} and the fact that
the function $u\mapsto\frac{\log(u^2)}{1-u^2}$ is bounded on $[c_\alpha\,,1]$, it follows that the series
$$\frac1{1-u^2}\sum_{j=1}^\infty\frac{(\log(u^2))^j}{j!\cdot(j+k+1)}B_{j+k+1}\alpha^{j+k+1}
$$
converges uniformly in $u\in[c_\alpha\,,1]$. By~\eqref{eq:hka-hat}, the corresponding limit function is
$$u\mapsto-\alpha\hat h_{k,\alpha}(1-u^2)-\hat\Phi_k(1-u^2)
$$
with~$\hat\Phi_k$ from Definition~\ref{def:powera}(ii).
In particular,
$$\sum_{j=1}^\infty\frac{\int_{c_\alpha}^1\tau_j(u)\,du}{j!\cdot(j+k+1)}B_{j+k+1}\alpha^{j+k+1}=-\int_{c_\alpha}^1(\alpha\hat h_{k,\alpha}(1-u^2)+\hat\Phi_k(1-u^2))\,du.
$$
Hence,
\begin{equation*}
\begin{aligned}
D^*(r,\alpha)&=-\sum_{j=r+1}^\infty\frac{\int_{c_\alpha}^1\tau_j(u)\,du}{j!\cdot(j+k+1)}B_{j+k+1}\alpha^{j+k+1}\\
&=-\alpha^{r+k+2}\sum_{j=r+1}^\infty\frac{\int_{c_\alpha}^1\tau_j(u)\,du}{j!\cdot(j+k+1)}B_{j+k+1}\alpha^{j-r-1}.
\end{aligned}
\end{equation*}
This is in $o(\alpha^{r+k+1})$, as claimed.
\end{proof}

\begin{corollary}
\label{cor:taylorF}
The function~$F$ from~\eqref{eq:defF} continuously extends to a function $F:[0,\infty)\,\to\mathbb{R}$. This function is infinitely differentiable at $\alpha=0$, and the
corresponding Taylor series $T_{0}F$ of~$F$ around~$0$ is given by
\begin{equation*}
\begin{aligned}
T_{0}F(\alpha)
={}&-\int_0^1\left(\hat\Phi_2(1-u^2)-\frac14\hat\Phi_0(1-u^2)\right)\,du\\
&{}+\frac{I_1}{48}\alpha^2
-\sum_{j=1}^\infty \left(\frac{I_j}{j!\cdot(j+3)}-\frac14\cdot\frac{I_{j+2}}{(j+3)!}\right)B_{j+3}\alpha^{j+3},
\end{aligned}
\end{equation*}
where $I_j$ is as in Proposition~\ref{prop:Fapprox}.
\end{corollary}

\begin{theorem}
\label{thm:taylorF}
The Taylor series $T_0F$ of~$F$ around~$0$ has convergence radius~$0$.
\end{theorem}

\begin{proof}
Since $B_{j+3}=0$ if $j\in\mathbb{N}$ is even, only the summands with $j$ odd actually occur in~$T_0F$.
For each $n\in\mathbb{N}$ one has (noting that $\log(u^2)=2\log(u)$)
$$I_{2n-1}=2^{2n-1}\frac{1-2^{2n}}{4n}\pi^{2n}|B_{2n}|;
$$
see, e.g., \cite{GR07}, p.~550. In particular, for each odd $j\in\mathbb{N}$,
$$I_j=2^j\cdot\frac{1-2^{j+1}}{2(j+1)}\pi^{j+1}|B_{j+1}| \text{ \ and \ }I_{j+2}=2^{j+2}\cdot\frac{1-2^{j+3}}{2(j+3)}\pi^{j+3}|B_{j+3}|.
$$
Let
$$V_j:=-\left(\frac{I_j}{j!\cdot(j+3)}-\frac14\cdot\frac{I_{j+2}}{(j+3)!}\right)
$$
for all $j\in\mathbb{N}$. Then for all odd $j\in\mathbb{N}$ we obtain:
\begin{equation*}
\begin{aligned}
V_j&=2^j\left(\frac{2^{j+1}-1}{2(j+1)!\cdot(j+3)}{\pi^{j+1}}|B_{j+1}|-\frac{2^{j+3}-1}{2(j+3)!\cdot(j+3)}\pi^{j+3}|B_{j+3}|\right)\\
&=\frac{2^{j-1}\pi^{j+3}}{(j+3)!\cdot(j+3)}\left((2^{j+1}-1)\frac{(j+2)(j+3)}{\pi^2}|B_{j+1}|-(4\cdot 2^{j+1}-1)|B_{j+3}|\right)\\
&=\frac{2^{j-1}\pi^{j+3}|B_{j+3}|}{(j+3)!\cdot(j+3)}\left((2^{j+1}-1)\frac{(j+2)(j+3)}{\pi^2}\cdot\frac{|B_{j+1}|}{|B_{j+3}|}-(2^{j+3}-1)\right)\\
&=\frac{2^{j-1}\pi^{j+3}|B_{j+3}|}{(j+3)!\cdot(j+3)}\left((2^{j+1}-1)D_j-(2^{j+3}-1)\right),
\end{aligned}
\end{equation*}
where
$$D_j:=\frac{(j+2)(j+3)}{\pi^2}\cdot\frac{|B_{j+1}|}{|B_{j+3}|}\,.
$$
Bagul~\cite{Ba23} recently proved that
$$\frac{2(2n)!}{\pi^{2n}(2^{2n}-1)}\cdot\frac{3^{2n}}{3^{2n}-\alpha}<|B_{2n}|<\frac{2(2n)!}{\pi^{2n}(2^{2n}-1)}\cdot\frac{3^{2n}}{3^{2n}-\beta}
$$
for all $n\in\mathbb{N}$, all $\alpha\le1$ and all $\beta\ge 9(1-\frac8{\pi^2})\approx 1.704875$. We use $\alpha:=1$, $\beta:=2$ in this formula and conclude, for each odd $j\in\mathbb{N}$:
\begin{equation*}
D_j>\frac19\cdot\frac{(2^{j+3}-1)(3^{j+3}-2)}{(2^{j+1}-1)(3^{j+1}-1)}=\frac{(2^{j+3}-1)(3^{j+1}-\frac29)}{(2^{j+1}-1)(3^{j+1}-1)}
\end{equation*}
and
\begin{equation*}
\begin{aligned}
(2^{j+1}-1)D_j-(2^{j+3}-1)&>\frac{(2^{j+3}-1)(3^{j+1}-\frac29)-(2^{j+3}-1)(3^{j+1}-1)}{3^{j+1}-1}\\
&=\frac79\cdot\frac{2^{j+3}-1}{3^{j+1}-1}
\end{aligned}
\end{equation*}
Consequently,
$$V_j>\frac{2^{j-1}\pi^{j+3}|B_{j+3}|}{(j+3)!\cdot(j+3)}\cdot\frac79\cdot\frac{2^{j+3}-1}{3^{j+1}-1}>0
$$
for each odd $j\in\mathbb{N}$.
Using $\limsup_{j\to\infty}\sqrt[j+3]{|B_{j+3}|/(j+3)!}=\frac1{2\pi}$, we conclude
\begin{equation}
\label{eq:vjsqrtj}
\limsup_{j\to\infty}\sqrt[j+3]{|V_j|}\ge\frac23>0.
\end{equation}
Recall from Corollary~\ref{cor:taylorF} that the coefficient of $\alpha^{j+3}$ in $T_0F(\alpha)$ is $V_j\cdot B_{j+3}$ for each odd $j\in\mathbb{N}$. Now~\eqref{eq:vjsqrtj} implies that 
$$\limsup_{j\to\infty}\sqrt[j+3]{|V_j|\cdot|B_{j+3}|}=\infty.
$$
Thus, the radius of convergence of $T_0F$ is indeed zero.
\end{proof}

\begin{proof}[Proof of Theorem~\ref{thm:irrat}]
By Theorem~\ref{thm:taylorF}, the continuous extension $F:[0,\infty) \to\mathbb{R}$ from Corollary~\ref{cor:taylorF} of our original $F:(0,\infty)\to\mathbb{R}$ cannot coincide with the restriction of any analytic function. Since rational functions are analytic, $F$ cannot be rational. Theorem~\ref{thm:irrat} now follows.
\end{proof}


\begin{thebibliography}{13}

\bibitem{Ba23}
Y.J. Bagul, \emph{Stringent bounds for the non-zero Bernoulli numbers},
preprint, 2023, https://arxiv.org/abs/2303.14532

\bibitem{BS87}
J. Br\"uning, R. Seeley, \emph{The resolvent expansion for second order regular singular operators},
\textrm{J. Funct. Anal.} \textbf{73} (1987), no. 2, 369--429.

\bibitem{BS91}
J. Br\"uning, R. Seeley, \emph{The Expansion of the Resolvent near a Singular Stratum of Conical Type},
\textrm{J. Funct. Anal.} \textbf{95} (1991), no. 2, 255--290.

\bibitem{Ch83}
J. Cheeger, \emph{Spectral geometry of singular Riemannian spaces},
\textrm{J. Differential Geom.} \textbf{18} (1983), no. 4, 575–-657.

\bibitem{DGGW}
E.B. Dryden, C.S. Gordon, S.J. Greenwald, and D.L. Webb,
\emph{Asymptotic expansion of the heat kernel for orbifolds},
\textrm{Michigan J. Math.} \textbf{56} (2008), no. 1, 205--238.

\bibitem{GR07}
I. S. Gradshteyn, I. M. Ryzhik, \emph{Table of integrals, series and products},
\textrm{Elsevier/Academic Press Inc.,~Amsterdam, seventh edition}, 2007.

\bibitem{HLV18}
L. Hartmann, M. Lesch, B. Vertman, \emph{On the domain of Dirac and Laplace type operators on stratified spaces},
\textrm{J. Spectr. Theory} \textbf{8} (2018), no. 4, 1295–-1348. 

\bibitem{HLV21}
L. Hartmann, M. Lesch, B. Vertman, \emph{Resolvent trace asymptotics on stratified spaces},
\textrm{Pure Appl. Anal.} \textbf{3} (2021), no. 1, 75--108.

\bibitem{Le97}
M. Lesch, \emph{Operators of Fuchs type, conical singularitieres, and asymptotic methods},
\textrm{Teubner Texte zur Mathematik}, Vol.~136, Teubner-Verlag, Leipzig, 1997.

\bibitem{Sc19}
D. Schueth, \emph{On the corner contributions to the heat coefficients of geodesic polygons},
\textrm{Ann. Inst. Fourier} \textbf{69} (2019), no. 7, 2827--2855.

\bibitem{Su17}
A. Suleymanova, \emph{On the spectral geometry of manifolds with conical singularities},
\textrm{PhD thesis} (2017), \textrm{Humboldt-Universit\"at zu Berlin, edoc-Server},
https://doi.org/10.18452/18420

\bibitem{Su17-2}
A. Suleymnova, \emph{Spectral geometry of surfaces with curved conic singularities},
preprint, 2017, https://arxiv.org/abs/1711.00577

\bibitem{Uc17}
E. U\c car, \emph{Spectral invariants for polygons and orbisurfaces},
\textrm{PhD thesis} (2017), \textrm{Humboldt-Universit\"at zu Berlin, edoc-Server},
https://doi.org/10.18452/18463 (also: https://arxiv.org/abs/1711.03405)

\end{thebibliography}
\end{document}